\documentclass[11pt,reqno]{amsart}
\usepackage{mathtools}
\usepackage{amsfonts}
\usepackage{amssymb,amsthm,amscd,array}
\usepackage{graphics}
\usepackage[final]{epsfig}
\usepackage[all]{xy}

\let\ssection=\section
\renewcommand{\section}{\setcounter{equation}{0}\ssection}

\newcommand {\emptycomment}[1]{}

\newcommand{\bbZ}{\mathbb{Z}}

\newcommand{\Hom}{\mathrm{Hom}}
\newcommand{\Ext}{\mathrm{Ext}}

\newcommand{\half}{\textstyle{\frac{1}{2}}}

\newcommand{\fa}{\mathfrak{a}}
\newcommand{\fb}{\mathfrak{b}}

\chardef\s=110
\chardef\g=103

\newtheorem{thm}{Theorem}
\newtheorem{lemma}{Lemma}[section]

\newtheorem{prop}[lemma]{Proposition}

\newtheorem{exe}[lemma]{Example}
\newtheorem{defi}[lemma]{Definition}

\def\a{\alpha}
\def\b{\beta}

\def\e{\varepsilon}
\def\g{\gamma}

\def\r{\rho}

\def\wt{\widetilde}

\begin{document}

\title{Universal central extensions of Hom-Lie antialgebras}

\author[T. Zhang]{Tao Zhang}
\address{College of Mathematics and Information Science\\
Henan Engn Lab Big Data Stat Anal \& Optimal Contr\\
Henan Normal University\\
Xinxiang 453007, PR China}
\email{zhangtao@htu.edu.cn}

\author[D. Zhong]{Deshou Zhong}
\address{School of Economics\\
University of Chinese Academy of Social Sciences\\
No.11,Changyu street, Fangshan District, Beijing 102488, PR China}
\email{zhongdeshou@cass.org.cn}


\date{}

\begin{abstract}
  We develop a theory of universal central extensions for Hom-Lie antialgebras.
  It is proved that a Hom-Lie antialgebra admits a universal central extension if and only if it is perfect.
  Moreover, we show that the kernel of the universal central extension is equal to the second homology group with trivial coefficients.
\end{abstract}

\subjclass[2010]{Primary 17D99; Secondary 18G60}

\keywords{Hom-Lie antialgebra; cohomology; universal central extensions}

\maketitle

\thispagestyle{empty}


\section{Introduction}

This is the third of our series papers which are denoted to study the theory of Hom-Lie antialgebras.
The notion of Lie antialgebras was introduced by V. Ovsienko in \cite{Ovs} as an algebraic structure in the context of symplectic and contact geometry of $\bbZ_2$-graded space.  A Lie antialgebra is a $\bbZ_2$-graded vector space $\fa=\fa_0\oplus\fa_1$ where $\fa_0$ is a commutative associative algebra acting on $\fa_1$ as a derivation satisfying some compatible conditions, see also \cite{LM,LO,MG}.
In the first paper \cite{ZZ1}, we introduced the concept of Hom-Lie antialgebras which is a Hom-analogue of Ovsienko's Lie antialgebras.
The general representations and cohomology theory of Hom-Lie antialgebras are developed. For cohomology theory of other types of algebras, see \cite{AEM,HLS,Sheng,Yau,Zhang1,Zhang2}.
In the second paper \cite{ZZ2}, we introduced the notation of crossed module for Hom-Lie antialgebras.
The relationship between the crossed module extensions of Hom-Lie antialgebras and the third cohomology group are investigated.

There arise a question that  can we provide some general construction for nontrivial examples of crossed modules for Hom-Lie antialgebras?
Under what conditions can we  obtain a nonzero  cohomology group?
The present investigation of this paper is devoted to solve these problems.

The problems are partial  solved by using central extension  and universal central extension theory.
It is well known that  universal central extension is an important construction in Mathematics and Physics.
The famous Virasoro algebra is a central extension of the Lie algebra of vector fields  on the circle. 
For the universal central extension of Lie groups,  Loop groups, Lie algebras, Leibniz algebras and  Lie-Rinehart algebras,  see \cite{Neeb, Gar, Gne, Cas, CC, CGL, Loday, WLX}.
For the universal central extension of different Hom-type of algebras, see \cite{CIP, CIP2, CS}.
For the  extension of (Hom-)Lie superalgebras,  see \cite{AF, WPD}.
It is a natural question that if the universal central extension construction can be developed in the theory of Hom-Lie antialgebras.
Thanks to the cohomology theory of Hom-Lie antialgebras developed in \cite{ZZ1}, we are able to introduce the concept of a perfect Hom-Lie antialgebra and construct a canonical one.
Using this, we prove that a Hom-Lie antialgebra admits a universal central extension if and  only if it is perfect.
Moreover, we show that the kernel of the universal central extension is equal to the second homology group with trivial coefficients.

The paper is organized as follows. In Section 2, we give some notations and properties of Hom-Lie antialgebras.
 In Section 3, we define the low dimensional (co)homology theory of Hom-Lie antialgebras. It is proved that very central  extension of Hom-Lie antialgebras give rise to a crossed module.
In Section 4, we study the universal central extension theory of Hom-Lie antialgebras.
We prove our main result that a Hom-Lie antialgebra admits a universal central extension if and only if it is perfect.

Throughout this paper, we work over an algebraically closed field  of characteristic 0.
For a  $\bbZ_2$-graded vector space (vector superspace) $V=V_0\oplus V_1$, we denote by $(V\otimes V)_0=(V_0 \otimes V_0)\oplus (V_1\otimes V_1)$ and $(V\otimes V)_1=(V_0\otimes V_1)\oplus(V_1\otimes V_0)$. A morphism $f=(f_0, f_1) : V \to W$ between two $\bbZ_2$-graded vector spaces $V$ and $W$ is assumed to be grade-preserving linear maps: $f_0:V_0\to W_0$ and $f_1:V_1\to W_1$. A Hom-vector superspace $(V,\a_{V_0},\b_{V_1})$ is a $\bbZ_2$-graded vector space $V=V_0\oplus V_1$ with two linear maps $\a_{V_0}:V_0\to V_0,~\b_{V_1}:V_1\to V_1$.

\section{Hom-Lie antialgebras}\label{defintion}

\begin{defi}\cite{ZZ1}\label{def1}
 A Hom-Lie antialgebra $(\fa,\a,\b)$ is a supercommutative $\bbZ_2$-graded algebra:
$\fa=\fa_0\oplus\fa_1$, together with three structural operations: $\cdot:\fa_0\times\fa_0\to \fa_0, \, \cdot:\fa_0\times\fa_1 \to\fa_1$, $[,]:\fa_1\times\fa_1 \to\fa_0$ and two linear maps $\alpha:\fa_0\rightarrow \fa_0,~\beta:\fa_1\rightarrow \fa_1$,
satisfying the following identities:
\begin{eqnarray}
\label{hanti01}
\a(x_1)\cdot\left(x_2\cdot x_3\right)&=&\left(x_1\cdot x_2\right)\cdot\a(x_3),\\
\label{hanti02}
\a(x_1)\cdot(x_2\cdot y_1)&=&\half(x_1\cdot x_2)\cdot\b(y_1),\\
\label{hanti03}
\a(x_1)\cdot[y_1,y_2]&=&[x_1\cdot y_1,\b(y_2)]\;+\;[\b(y_1),x_1\cdot y_2],\\
\label{hanti04}
\b(y_1)\cdot[y_2,y_3]&+&\b(y_2)\cdot[y_3,y_1]\;+\;\b(y_3)\cdot[y_1,y_2]=0,
\end{eqnarray}
for all $x_1,x_2,x_3\in\fa_0$ and $y_1,y_2,y_3\in\fa_1$.
\end{defi}
In the above definition $(\fa,\a,\b)$ is a supercommutative algebra means that $x_1\cdot x_2=x_2\cdot x_1$, $x_1\cdot y_1=y_1\cdot x_1$ and $y_1\cdot y_2=-y_2\cdot y_1$. That is why we denote $y_1\cdot y_2$ by $[y_1,y_2]$ in the above equation  which is slightly different from notations in \cite{Ovs,LM}.

\begin{exe}\cite{ZZ1}
Consider the 3-dimensional Hom-Lie antialgebra $K_3$ as follows.
This algebra has the basis $\{\e; a,b\}$, where $\e$ is even and $a,b$ are odd,
Consider the linear map $(\a,\b):\fa\to \fa$ defined by
$$\a(\e)=\e, \quad \b(a)=\mu a,\quad \b(b)=\mu^{-1} b$$
on the basis elements, then we obtain a Hom-Lie antialgebra structure given by
\begin{eqnarray*}
\label{aslA}
\e\cdot{}\e=\e,\quad
\e\cdot{}a=\half\,\mu\, a,\quad
\e\cdot{}b=\textstyle{\frac{1}{2}}\mu^{-1} b,\quad
{[a,b]}=\half\,\e.
\end{eqnarray*}
\end{exe}

\begin{exe}\cite{ZZ1}
\label{ExMain}
Another example of a Hom-Lie antialgebra is the
\textit{conformal Hom-Lie antialgebra}  $K(1)$.
This is a simple infinite-dimensional Hom-Lie antialgebra with the basis
$$
\textstyle
\left\{
\e_n,\;n\in\bbZ;
~a_i, ~i\in\bbZ+\half
\right\},
$$
where $\e_n$ are even, $a_i$ are odd, and
$\a(\e_i)=\e_i$, $\b(a_i)=(1+q^i) a_i$
satisfy the following relations:
\begin{eqnarray*}
\label{GhosRel}
\e_n\cdot{}\e_m&=&\e_{n+m},\\
\e_n\cdot{}a_i &=&\half (1+q^i) a_{n+i},\\
{[a_i, a_j]}&=&\half\left(\{j\}-\{i\}\right)\e_{i+j},
\end{eqnarray*}
where $\{i\}=(q^i-1)/(q-1),\  q\neq 1$.
\end{exe}

A Hom-Lie antialgebra is called a multiplicative Hom-Lie antialgebra if $\a,~\b$ are algebraic morphisms, i.e. for any $x_i\in\fa_0$, $y_i\in\fa_1$, we have
\begin{eqnarray*}
\a(x_1\cdot x_2)&=&\a(x_1)\cdot\a(x_2),\\
\b(x_1\cdot y_1)&=&\a(x_1)\cdot\b(y_1),\\
\a([y_1,y_2])&=&[\b(y_1),\b(y_2)].
\end{eqnarray*}


\begin{defi}
Let $(\fa,\alpha,\beta)$ be a Hom-Lie antialgebra. A Hom-Lie sub-antialgebra $\fb$ is a subspace of $\fa$, which is closed for the structural operations and invariant by $\alpha,\beta$, that is,
\begin{enumerate}
\item [(i)] $\fb_0 \cdot \fb_0\subseteq\fb_0, \fb_0 \cdot \fb_1\subseteq\fb_1, [\fb_1,\fb_1]\subseteq \fb_0$,
\item [(ii)] $\alpha(\fb_0)\subseteq \fb_0, \beta(\fb_1)\subseteq \fb_1$.
\end{enumerate}
We call a Hom-Lie sub-antialgebra $\fb$ of $\fa$ to be an ideal  if
\begin{enumerate}
\item [(iii)] $\fb_0 \cdot \fa_0\subseteq\fb_0,\, \fb_0 \cdot \fa_1\subseteq\fb_1,\, \fb_1 \cdot \fa_0\subseteq\fb_1$ and $[\fb_1, \fa_1]\subseteq \fb_0$,
\end{enumerate}
where we denote by $\fb_0 \cdot \fb_0=\{x_1\cdot x_2|\forall x_1, x_2\in \fb_0\}$, $\fb_0 \cdot \fa_1=\{x_1\cdot y_1|\forall x_1\in \fb_0, \forall y_1\in \fa_1\}$, $[\fb_1, \fa_1]=\{[y_1,y_2]|\forall y_1\in \fb_1,y_2\in \fa_1\}$, etc.
\end{defi}

\emptycomment{
\begin{defi}
The center of a  Hom-Lie antialgebra $(\fa,\alpha,\beta)$ is the subspace
$$Z(\fa) = \{ (x_1,y_1) \in \fa \mid x_1 \cdot x_2 =0, x_1 \cdot y_2=0, [y_1,y_2]=0, \forall x_2, y_2\in \fa\}$$
\end{defi}
It is easy to see that center $Z(\fa)$ is a Hom-Lie sub-antialgebra of a  Hom-Lie antialgebra $(\fa,\alpha,\beta)$.

When $\a=\b$, we simply call it $\a$-Hom-Lie antialgebra.

An equivalent definition is as follows.
A Hom-$\bbZ_2$-graded commutative algebra $\fa$ is a Hom-Lie antialgebra if and only if the following three conditions are satisfied.

\begin{enumerate}
\item
The subalgebra $\fa_0\subset\fa$ is Hom-associative.

\item
$\fa_0$ acts Hom-commutatively on $\fa_1$:
$\a(x_1)\cdot(x_2\cdot y)=\a(x_2)\cdot(x_1\cdot y)$.

\item
For every $y_i\in{}\fa_1$, the operator of right multiplication by $y$
is an \textit{odd} derivation of $\fa$.
\end{enumerate}
}

\begin{defi}
Let $(\fa,\a,\b)$ and $(\fa',\a',\b')$ be two Hom-Lie antialgebras. A Hom-Lie antialgebra homomorphism $\phi:\fa\to \fa'$ consists of  linear maps $\phi_0:\fa_0\rightarrow\fa'_0$, $\phi_1:\fa_1\rightarrow\fa'_1$, such that the following equalities hold for all $x_1,x_2\in\fa_0,~y_1,y_2\in\fa_1$:
\begin{eqnarray}
\label{homo01}
\phi_0\circ\a&=&\a'\circ\phi_0,\\
\label{homo02}
\phi_1\circ\b&=&\b'\circ\phi_1,\\
\label{homo1}
\phi_0(x_1\cdot x_2)&=&\phi_0(x_1)\cdot\phi_0(x_2),\\
\label{homo2}
\phi_1(x_1\cdot y_1)&=&\phi_0(x_1)\cdot\phi_1(y_1),\\
\label{homo3}
\phi_0([y_1,y_2])&=&[\phi_1(y_1),\phi_1(y_2)],
\end{eqnarray}
where in the right hand side of \eqref{homo1}--\eqref{homo3} the structural operations $\cdot'$ and $[,]'$ in $(\fa',\a',\b')$ are denoted by $\cdot$ and $[,]$ for simplicity  if it does not cause confusion.
\end{defi}

Given two Hom-Lie antialgebras $(\fa,\alpha,\beta)$ and
$(\fa',\alpha',\beta')$, there is a Hom-Lie antialgebra
$(\fa\oplus\fa',\alpha+\alpha',\beta+\beta')$,
where the structural operations are given by
\begin{eqnarray*}
 {(x_1,x'_1)\cdot(x_2,x'_2)}&=&(x_1\cdot x_2,x'_1\cdot x'_2),\\
 {(x_1,x'_1)\cdot(y_1,y'_1)}&=&(x_1\cdot y_1,x'_1\cdot y'_1),\\
  { [(y_1,y'_1),(y_2,y'_2)]}&=&([y_1,y_2],[y'_1,y'_2]),
\end{eqnarray*}
and the linear maps $\alpha+\alpha',\beta+\beta':\fa\oplus\fa'\to \fa\oplus\fa'$ are given by
\begin{eqnarray*}
(\alpha+\alpha')(x_1,x_1')&=&(\alpha(x_1),\alpha'(x_1')),\\
(\beta+\beta')(y_1,y_1')&=&(\beta(y_1),\beta'(y_1')),
\end{eqnarray*}
for all $x_1,x_2\in\fa_0, ~y_1,y_2\in\fa_1$ and $x'_1,x'_2\in\fa'_0, ~y'_1,y'_2\in\fa'_1$.

Denote by $G^\phi_0=\{(x,\phi_0(x)\}\subset\fa_0\oplus\fa'_0, G^\phi_1=\{(y,\phi_1(y)\}\subset\fa_1\oplus\fa'_1$  the graph of linear maps $\phi_0:\fa_0\to \fa'_0$, $\phi_1:\fa_1\to \fa'_1$. Put $G^\phi= G^\phi_0\oplus G^\phi_1\subseteq\fa\oplus\fa'$.

\begin{prop}
Let $(\fa,\alpha,\beta)$ and $(\fa',\alpha',\beta')$ be two Hom-Lie antialgebras.
 A linear map $\phi:\fa\to \fa'$ is a homomorphism of Hom-Lie antialgebras if and only if the graph
$G^\phi$ is a Hom-Lie sub-antialgebra of $(\fa\oplus\fa',\alpha+\alpha',\beta+\beta')$.
\end{prop}

\begin{proof}
Let $\phi:\fa\to \fa'$ be a homomorphism of Hom-Lie antialgebras, then for any $y_1,y_2\in\fa$, we have
\begin{eqnarray*}
 {(x_1,\phi_0(x_1))\cdot(x_2,\phi_0(x_2))}&=&(x_1\cdot x_2,\phi_0(x_1)\cdot \phi_0(x_2))=(x_1\cdot x_2,\phi_0(x_1\cdot x_2)),\\
 {(x_1,\phi_0(x_1))\cdot(y_1,\phi_1(y_1))}&=&(x_1\cdot y_1,\phi_0(x_1)\cdot \phi_1(y_1))=(x_1\cdot y_1,\phi_0(x_1\cdot y_1)),\\
 {[(y_1,\phi_1(y_1)),(y_2,\phi(y_2))]}&=&([y_1,y_2],[\phi_1(y_1),\phi(y_2)])=([y_1,y_2],\phi[y_1,y_2]).
\end{eqnarray*}
Thus the graph $ G^\phi$ is closed under the structural operations. Furthermore, by
\eqref{homo01} and \eqref{homo02}, we have
\begin{eqnarray*}
(\alpha+\alpha')(x,\phi_0(x))=(\alpha(x),\alpha'\circ\phi_0(x))=(\alpha(x),\phi_0\circ\alpha(x)),\\
(\beta+\beta')(y,\phi_1(y))=(\beta(y),\beta'\circ\phi_1(y))=(\beta(y),\phi_1\circ\beta(y)),
\end{eqnarray*}
which implies that $(\alpha+\alpha')(G^\phi_0)\subset G^\phi_0$ and $(\beta+\beta')(G^\phi_1)\subset G^\phi_1$.
Thus $ G^\phi$ is a Hom-Lie sub-antialgebra of
$(\fa\oplus\fa',\alpha+\alpha',\beta+\beta')$

Conversely, if graph $G^\phi$ is a Hom-Lie sub-antialgebra of $(\fa\oplus\fa',\alpha+\alpha',\beta+\beta')$,
then we have
\begin{eqnarray*}
 {(x_1,\phi_0(x_1))\cdot(x_2,\phi_0(x_2))}&=&(x_1\cdot x_2,\phi_0(x_1)\cdot \phi_0(x_2))\in G^\phi_0,\\
 {(x_1,\phi_0(x_1))\cdot(y_1,\phi_1(y_1))}&=&(x_1\cdot y_1,\phi_0(x_1)\cdot \phi_1(y_1))\in G^\phi_1,\\
 {[(y_1,\phi_1(y_1)),(y_2,\phi(y_2))]}&=&([y_1,y_2],[\phi_1(y_1),\phi(y_2)])\in G^\phi_0,
\end{eqnarray*}
which implies that
\begin{eqnarray*}
 \phi_0(x_1)\cdot\phi_0(x_2)&=& \phi_0(x_1\cdot x_2),\\
 \phi_0(x_1)\cdot\phi_1(y_1)&=& \phi_1(x_1\cdot y_1),\\
 {[\phi_1(y_1),\phi_1(y_2)]}&=&\phi_0([y_1,y_2]).
\end{eqnarray*}
Furthermore, $(\alpha+\alpha')( G^\phi_0)\subset G^\phi_0$  and  $(\beta+\beta')( G^\phi_1)\subset G^\phi_1$ yields that
\begin{eqnarray*}
(\alpha+\alpha')(x,\phi(x))&=&(\alpha(x),\alpha'\circ\phi(x))\in G^\phi_0,\\
(\beta+\beta')(y,\phi(y))&=&(\beta(x),\beta'\circ\phi(y))\in G^\phi_1,
\end{eqnarray*}
which is equivalent to the condition
$\alpha'\circ\phi(x)=\phi\circ\alpha(x)$ and $\b'\circ\phi_1(y)=\phi_1\circ\b(y)$.
Therefore, $\phi:\fa\to \fa'$ is a homomorphism of
Hom-Lie antialgebras.
\end{proof}

\begin{defi}\cite{ZZ2}\label{def-action}
Let $(V,\alpha_{V_0},\beta_{V_1})$ and $(\fa,\alpha,\beta)$ be two Hom-Lie antialgebras.
An action of $(\fa,\alpha,\beta)$ over $(V,\alpha_{V_0},\beta_{V_1})$  is a representation
$(V,\rho)$ such that the following conditions hold:
\begin{eqnarray}
\label{action01}
\rho_0(\alpha(x_1))(u_1\cdot u_2)&=&\rho_0(x_1)(u_1)\cdot\alpha(u_2),\\
\label{action021}
\rho_0(\alpha(x_1))(u_1\cdot v_1)&=&\half\rho_0(x_1)(u_1)\cdot\beta_{V_1}(v_1),\\
\label{action022}
\rho_1(y_1)(u_2)\cdot\alpha_{V_0}(u_1)&=&\half\rho_1(\beta(y_1))(u_1\cdot u_2),\\
\label{action023}
\rho_0(x_1)(v_1)\cdot\alpha_{V_0}(u_1)&=&\half\rho_0(x_1)(u_1)\cdot\beta_{V_1}(v_1),\\
\label{action031}
\rho_0(\alpha(x_1))([v_1,v_2])&=&[\rho_0(x_1)(v_1),\beta_{V_1}(v_2)]+[\beta_{V_1}(v_1),\rho_0(x_1)(v_2)],\\
\label{action032}
\rho_1(\beta(y_1))(u_1\cdot v_1)&=&\alpha_{V_0}(u_1)\cdot\rho_1(y_1)(v_1)-[\rho_1(y_1)(u_1),\beta_{V_1}(v_1)],\\
\label{action04}
\rho_1(\beta(y_1))([v_1,v_2])&=&\beta_{V_1}(v_1)\cdot\rho_1(y_1)( v_2)-\beta_{V_1}(v_2)\cdot \rho_1(y_1)(v_1),
\end{eqnarray}
for all $x_1\in\fa_0,~y_1\in\fa_1$, $u_1,u_2\in V_0,~v_1,v_2\in V_1$.
\end{defi}

\begin{thm}\label{Thm1}\cite{ZZ2}
Let $(V,\alpha_{V_0},\beta_{V_1})$ and $(\fa,\alpha,\beta)$ be two Hom-Lie antialgebras with an action of $(\fa,\alpha,\beta)$ over $(V,\alpha_{V_0},\beta_{V_1})$.
 Then $(\fa\oplus V,\a+\a_{V_0},\b+\b_{V_1})$ is a Hom-Lie antialgebra under the following maps:
\begin{eqnarray}
\label{homo1}(\a+\a_{V_0})(x_1,u_1)&=&(\a(x_1),\a_{V_0}(u_1)),\notag\\
\label{homo2}(\b+\b_{V_1})(y_1,v_1)&=&(\b(y_1),\b_{V_1}(v_1)),\notag\\
\label{op1}
(x_1,u_1)\cdot(x_2,u_2)&=&(x_1\cdot x_2,~ \r_0(x_1)(u_2)+\r_0(x_2)(u_1)
+u_1\cdot u_2),\notag\\
\label{op2}
(x_1,u_1)\cdot(y_1,v_1)&=&(x_1\cdot y_1,~ \r_0(x_1)(v_1)+\r_1(y_1)(u_1)
+u_1\cdot v_1),\notag\\
\label{op3}
[(y_1,v_1),(y_2,v_2)]&=&([y_1,y_2],~ \r_1(y_1)(v_2)-\r_1(y_2)(v_1)
+[v_1,v_2]),\notag
\end{eqnarray}
for all $x_1,x_2\in\fa_0,~y_1,y_2\in\fa_1, u_1,u_2\in V_0, v_1,v_2\in V_1$.
This is called a semidirect product of $\fa$ and $V$,
denoted by $\fa\ltimes V$.
\end{thm}

\begin{defi}\label{def:crossedmod}\cite{ZZ2}
 Let $(V,\alpha_{V_0},\beta_{V_1})$ and $(\fa,\alpha,\beta)$ be two Hom-Lie antialgebras with an action of  $(\fa,\alpha,\beta)$ over $(V,\alpha_{V_0},\beta_{V_1})$.
A crossed module of Hom-Lie antialgebras $\partial:(V,\alpha_{V_0},\beta_{V_1})\rightarrow(\fa,\alpha,\beta)$ is an Hom-Lie antialgebra homomorphism such that the following identities hold:
\begin{eqnarray}
\label{cm1}
\partial_0\circ\rho_0(x_1)(u_1)&=&x_1\cdot\partial_0(u_1),\\
\label{cm2}
\partial_1\circ\rho_0(x_1)(v_1)&=&x_1\cdot\partial_1(v_1),\\
\label{cm3}
\partial_1\circ\rho_1(y_1)(u_1)&=&y_1\cdot\partial_0(u_1),\\
\label{cm4}
\partial_0\circ\rho_1(y_1)(v_1)&=&[y_1,\partial_1(v_1)],\\
\label{pei1}
\rho_0(\partial_0(u_1))(u_2)&=&u_1\cdot u_2,\\
\label{pei2}
\rho_1(\partial_1(v_1))(v_2)&=&[v_1,v_2],\\
\label{pei3}
\rho_0(\partial_0(u_1))(v_1)&=&\rho_1(\partial_1(v_1))(u_1)=u_1\cdot v_1,
\end{eqnarray}
for all $x_1\in\fa_0$, $u_1,u_2\in V_0$, $y_1\in\fa_1$, $v_1,v_2\in V_1$.
\end{defi}

We will denote a crossed module by  $(V,\fa, \partial)$  or $\partial:V\rightarrow\fa$ for simplicity.

\section{Low (co)homology and central extensions}\label{sect:cohomology}


Now we introduce some low dimensional homology and cohomology theory for Hom-Lie antialgebra with trivial coefficients.

Let $(\fa,\a,\b)$ be a Hom-Lie antialgebra and denote by $C^n(\fa)=\otimes^n\fa^*$ the $n$-cochain where $\fa^*=\fa_0^*\oplus\fa_1^*$ is dual space of $\fa$. We define the linear maps
$$d^n : C^{n-1}(\fa)\to C^{n}(\fa)(n=1,2)$$
by
$$d^1:\fa^*\to \fa^*\otimes\fa^*,\quad d^2:\fa^*\otimes\fa^*\to \fa^*\otimes\fa^*\otimes\fa^*$$
as follows.

For 1-cochain
\begin{eqnarray*}
\upsilon=(\upsilon_0,\upsilon_1)\in \fa^*,
\end{eqnarray*}
where $\upsilon_0\in\fa_0^*,\,\upsilon_1\in\fa_1^*$, we define
\begin{eqnarray}
\label{Thomo1}
d^1\upsilon(x_1,x_2)&=&\upsilon_0(x_1\cdot x_2),\\
\label{Thomo2}
d^1\upsilon(x_1,y_1)&=&\upsilon_1(x_1\cdot y_1),\\
\label{Thomo3}
d^1\upsilon(y_1,y_2)&=&\upsilon_0([y_1,y_2]).
\end{eqnarray}
For 2-cochain
\begin{eqnarray*}
\omega=(\omega_0,\omega_1,\omega_2)\in \fa^*\otimes\fa^*,
\end{eqnarray*}
where $\omega_0\in\fa_0^*\otimes \fa_0^*,\quad \omega_1\in\fa_0^*\otimes \fa_1^*,\quad \omega_2\in\fa_1^*\otimes \fa_1^*$, we define
\begin{eqnarray}
d^2\omega(x_1,x_2,x_3)&=&\omega_0(\a(x_1),x_2\cdot x_3)-\omega_0(x_1\cdot x_2,\a(x_3)),\\
\label{cocycle2}
d^2\omega(x_1,x_2,y_1)&=&\omega_1(\a(x_1),x_2\cdot y_1)-\half\omega_1(x_1\cdot x_2,\b(y_1)),\\
\label{cocycle3}
\notag d^2\omega(x_1,y_1,y_2)&=&\omega_0(\a(x_1),[y_1,y_2])-\omega_2(x_1\cdot y_1,\b(y_2))\\
&&-\omega_2(\b(y_1),x_1\cdot y_2),\\
\label{cocycle4}
\notag d^2\omega(y_1,y_2,y_3)&=&\omega_1(\b(y_1),[y_2,y_3])+\omega_1(\b(y_2),[y_3,y_1])\\
&&+\omega_1(\b(y_3),[y_1,y_2]).
\end{eqnarray}

\begin{prop}\label{prop:dd=0}
With the above notations, we have $d^{2}\circ d^1=0$.
\end{prop}

\begin{proof} By direct computations, we get
\begin{eqnarray*}
&&d^{2}\circ d^1\upsilon(x_1,x_2,x_3)\\
&=&d^1\upsilon_0(\a(x_1),x_2\cdot x_3)-d^1\upsilon_0(x_1\cdot x_2,\a(x_3))\\
&=&\upsilon_0(\a(x_1)\cdot(x_2\cdot x_3)-(x_1\cdot x_2)\cdot\a(x_3))\\
&=&0.
\end{eqnarray*}
Similarly, we have
\begin{eqnarray*}
d^{2}\circ d^1\upsilon(x_1,x_2,y_1)&=&\upsilon_1(\a(x_1)\cdot(x_2\cdot y_1)-\half(x_1\cdot x_2)\cdot\b(y_1))=0,\\
d^{2}\circ d^1\upsilon(x_1,y_1,y_2)&=&\upsilon_0(\a(x_1)\cdot[y_1,y_2])-\upsilon_0([x_1\cdot y_1,\b(y_2)])\\
&&+\upsilon_0([\b(y_1),x_1\cdot y_2])=0,\\
d^{2}\circ d^1\upsilon(y_1,y_2,y_3)&=&\upsilon_1(\b(y_1)\cdot[y_2,y_3])+\upsilon_1(\b(y_2)\cdot[y_3,y_1])\\
&&+\upsilon_1(\b(y_3)\cdot[y_1,y_2])=0.
\end{eqnarray*}
This complete the proof.
\end{proof}

Denote by $C_n(\fa)=\otimes^n\fa$. We define the linear maps
$$d_n : C_n(\fa)\to C_{n-1}(\fa)(n=2,3)$$
$$d_3:\fa\otimes\fa\otimes\fa\to \fa\otimes\fa,\quad d_2:\fa\otimes\fa\to \fa$$
by the following
\begin{align}
d_2(x_1\otimes x_2)&= x_1\cdot x_2,\\
d_2(x_1\otimes y_1)&= x_1\cdot y_1,\\
d_2(y_1\otimes y_2)&= [y_1, y_2],
\end{align}
\begin{eqnarray}\label{cocycle1}
 d_3(x_1 \otimes x_2\otimes x_3)&=&\a(x_1)\otimes x_2\cdot x_3-x_1\cdot x_2 \otimes \a(x_3),\\
\label{cocycle2}
 d_3(x_1 \otimes  x_2 \otimes   y_1)&=&\a(x_1) \otimes  x_2\cdot y_1-\half x_1\cdot x_2 \otimes \b(y_1),\\
\label{cocycle3}
\nonumber d_3(x_1 \otimes y_1 \otimes  y_2)&=&\a(x_1) \otimes [y_1,y_2]-x_1\cdot y_1  \otimes  \b(y_2)\\
&&-\b(y_1) \otimes x_1\cdot y_2,\\
\label{cocycle4}
\nonumber d_3( y_1\otimes y_2\otimes y_3)&=& \b(y_1)\otimes[y_2,y_3] +\b(y_2)\otimes[y_3,y_1]\\
&&+\b(y_3)\otimes[y_1,y_2].
\end{eqnarray}

Since $(\fa,\a,\b)$ is a Hom-Lie antialgebra, we have the following result.
\begin{prop}\label{prop212}
With the above notations, we have $d_{2}\circ d_3=0$.
\end{prop}
\begin{proof}
The condition $d_{2}\circ d_3=0$ is just equivalent to the conditions \eqref{hanti01}--\eqref{hanti04} in the Definition \ref{def1}. So the proof is trivial.
\end{proof}

In the above, we have define the (co)homology for a Hom-Lie antialgebra $(\fa,\a,\b)$ with trivial coefficients.
In fact, we can also define (co)homology for a Hom-Lie antialgebra $(\fa,\a,\b)$ with nontrivial coefficients, see \cite{ZZ1} for details.

Let $(\fa,\a,\b)$ be a Hom-Lie antialgebra, $(V,\a_{V_0},\b_{V_1})$ be a Hom-vector superspace.
A triple of linear maps $\omega=(\omega_0,\omega_1,\omega_2)$ where
\begin{eqnarray*}
\omega_0:\fa_0\times \fa_0\to V_0,\quad \omega_1:\fa_0\times \fa_1\to V_1,\quad \omega_2:\fa_1\times \fa_1\to V_0
\end{eqnarray*}
is called  a 2-cocycle with coefficients in $V$ if they satisfy the following conditions:
\begin{eqnarray}\label{cocycle1}
\omega_0(\a(x_1),x_2\cdot x_3)-\omega_0(x_1\cdot x_2,\a(x_3))=0,\\
\label{cocycle2}
\omega_1(\a(x_1),x_2\cdot y_1)-\half\omega_1(x_1\cdot x_2,\b(y_1))=0,\\
\label{cocycle3}
\omega_0(\a(x_1),[y_1,y_2])-\omega_2(x_1\cdot y_1,\b(y_2))-\omega_2(\b(y_1),x_1\cdot y_2)=0,\\
\label{cocycle4}
\omega_1(\b(y_1),[y_2,y_3])+\omega_1(\b(y_2),[y_3,y_1])+\omega_1(\b(y_3),[y_1,y_2])=0.
\end{eqnarray}

Now we study the central extensions of Hom-Lie antialgebra.
We construct a central extension by using a 2-cocycle of Hom-Lie antialgebra $(\fa,\a,\b)$ with coefficients in $V$.

\begin{defi}
Let $(V,\a_{V_0},\b_{V_1})$, $(\fa,\a,\b)$ and $(\widetilde{\fa},\widetilde{\a},\widetilde{\b})$ be Hom-Lie antialgebras. An extension of $\fa$ by $V$ is a short exact sequence
\begin{equation*}
\xymatrix@C=0.5cm{
  0 \ar[r]^{} & V  \ar[rr]^{i } && \widetilde{\fa } \ar[rr]^{\pi} && \fa  \ar[r]^{} & 0 }
\end{equation*}
of Hom-Lie antialgebra. It is  called a central extension, if $V$ is contained in the center $Z(\widetilde{\fa})$ of $\widetilde{\fa}$.
\end{defi}

Let $(\fa,\a,\b)$ be a Hom-Lie antialgebra, $(V,\a_{V_0},\b_{V_1})$ be a Hom-vector superspace.
We define on the direct sum space $(\fa\oplus V,\a+\a_{V_0},\b+\b_{V_1})$ the following operations:
\begin{eqnarray}
(\alpha+\alpha_{V_0})(x_1,u_1)&:=&(\alpha(x_1),\alpha_{V_0}(u_1)),\notag\\
(\beta+\beta_{V_1})(y_1,v_1)&:=&(\beta(y_1),\alpha_{V_1}(v_1)),\notag\\
\label{op1}
(x_1,u_1)\cdot(x_2,u_2)&:=&(x_1\cdot x_2,\omega_0(x_1,x_2)),\notag\\
\label{op2}
(x_1,u_1)\cdot(y_1,v_1)&:=&(x_1\cdot y_1,\omega_1(x_1,y_1)),\notag\\
\label{op3}
[(y_1,v_1),(y_2,v_2)]&:=&([y_1,y_2], \omega_2(y_1,y_2)),\notag
\end{eqnarray}
for all $x_1,x_2\in\fa_0,~y_1,y_2\in\fa_1$, $u_1,u_2\in V_0,~w_1,w_2\in V_1$.
The space with the above operations is denoted by $\fa\oplus_\omega V$.
\begin{lemma}\label{lem01}
With the above notations, $(\fa\oplus_\omega V,\a+\a_{V_0},\b+\b_{V_1})$ is a Hom-Lie antialgebra if and only if $\omega=(\omega_0,\omega_1,\omega_2)$ is a $2$-cocycle of
Hom-Lie antialgebra $(\fa,\a,\b)$ with coefficients in $V$.
In this case, the following exact sequence
\begin{equation*}
\xymatrix@C=0.5cm{
  0 \ar[r]^{} & V  \ar[rr]^{i } && \fa\oplus_\omega V \ar[rr]^{\pi} && \fa  \ar[r]^{} & 0 }
\end{equation*}
is a central extension of $\fa$ by $V$, where $i$ is the canonical including map and $\pi$ is the projection map.
\end{lemma}

\begin{proof} Assume $\omega=(\omega_0,\omega_1,\omega_2)$ is a 2-cochain of Hom-Lie antialgebra $(\fa,\a,\b)$ with coefficients in $V$.
We will verify that $\fa\oplus_\omega V$ with the above operations satisfy identities \eqref{hanti01}--\eqref{hanti04} if and only if $\omega$ satisfy 2-cocycle conditions \eqref{cocycle1}--\eqref{cocycle4}.
Firstly, we compute
\begin{eqnarray*}
&&(\a+\a_{V_0})((x_1,u_1))\cdot((x_2,u_2)\cdot(x_3,u_3))\\
&=&(\a(x_1),\a_{V_0}(u_1))\cdot((x_2,u_2)\cdot(x_3,u_3))\\
&=&\big(\a(x_1)\cdot(x_2\cdot x_3),\omega_0(\a(x_1),x_2\cdot x_3)\big),
\end{eqnarray*}
\begin{eqnarray*}
&&((x_1,u_1)\cdot(x_2,u_2))\cdot(\a+\a_{V_0})((x_3,u_3))\\
&=&((x_1,u_1)\cdot(x_2,u_2))\cdot(\a(x_3),\a_{V_0}(u_3))\\
&=&\big((x_1\cdot x_2)\cdot \a(x_3),\omega_0(x_1\cdot x_2,\a(x_3)\big).
\end{eqnarray*}
Thus \eqref{hanti01} hold if and only if \eqref{cocycle1} hold.

Secondly, we compute
\begin{eqnarray*}
&&(\a+\a_{V_0})((x_1,u_1))\cdot((x_2,u_2)\cdot(y_1,v_1))\\
&=&(\a(x_1),\a_{V_0}(u_1))\cdot((x_2,u_2)\cdot(y_1,v_1))\\
&=&\big(\a(x_1)\cdot(x_2,y_1),\omega_1(\a(x_1),x_2\cdot y_1)\big),
\end{eqnarray*}
\begin{eqnarray*}
&&((x_1,u_1)\cdot(x_2,u_2))\cdot(\b+\b_{V_1})(y_1,v_1)\\
&=&((x_1,u_1)\cdot(x_2,u_2))\cdot(\b(y_1),\b_{V_1}(v_1))\\
&=&\big((x_1,x_2)\cdot\b(y_1),\omega_1(x_1\cdot x_2,\b(y_1))\big).
\end{eqnarray*}
Thus \eqref{hanti02} hold if and only if \eqref{cocycle2} hold.

Thirdly, we have
\begin{eqnarray*}
&&(\a+\a_{V_0})(x_1,u_1)\cdot[(y_1,v_1)\cdot(y_2,v_2)]\\
&=&(\a(x_1),\a_{V_0}(u_1))\cdot[(y_1,v_1)\cdot(y_2,v_2)]\\
&=&\big(\a(x_1)\cdot[y_1,y_2],\omega_0(\a(x_1),[y_1,y_2])\big),
\end{eqnarray*}
\begin{eqnarray*}
&&[(x_1,u_1)\cdot(y_1,v_1),(\b+\b_{V_1})(y_2,v_2)]\\
&=&[(x_1,u_1)\cdot(y_1,v_1)),(\b(y_2),\b_{V_1}(v_2))]\\
&=&\big((x_1\cdot y_1)\cdot\b(y_2),\omega_2(x_1\cdot y_1,\b_{V_1}(v_2))\big),
\end{eqnarray*}
\begin{eqnarray*}
&&[(\b+\b_{V_1})(y_1,v_1),(x_1,u_1)\cdot(y_2,v_2)]\\
&=&[(\b(y_1),\b_{V_1}(v_1)),(x_1,u_1)\cdot(y_2,v_2)]\\
&=&\big(\b(y_1)\cdot(x_1\cdot y_2),\omega_2(x_1\cdot y_2,\b(y_1))\big).
\end{eqnarray*}
Thus \eqref{hanti03} hold if and only if \eqref{cocycle3} hold.

Finally, we have
\begin{eqnarray*}
&&(\b+\b_{V_1})(y_1,v_1)\cdot[(y_2,v_2),(y_3,v_3)]\\
&=&(\b(y_1),\b_{V_1}(v_1))\cdot[(y_2,v_2),(y_3,v_3)]\\
&=&\big(\b(y_1)\cdot[y_2,y_3],\omega_1(\b(y_1),[y_2,y_3])\big),
\end{eqnarray*}
\begin{eqnarray*}
&&(\b+\b_{V_1})(y_2,v_2)\cdot[(y_3,v_3),(y_1,v_1)]\\
&=&(\b(y_2),\b_{V_1}(v_2))\cdot[(y_3,v_3),(y_1,v_1)\\
&=&\big(\b(y_2)\cdot[y_3,y_1],\omega_1(\b(y_2),[y_3,y_1])\big),
\end{eqnarray*}
\begin{eqnarray*}
&&(\b+\b_{V_1})(y_3,v_3)\cdot[(y_1,v_1),(y_2,v_2)]\\
&=&\big(\b(y_3),\b_{V_1}(v_3))\cdot[(y_1,v_1),(y_2,v_2)]\\
&=&\big(\b(y_3)\cdot[y_1,y_2],\omega_1(\b(y_3),[y_1,y_2])\big).
\end{eqnarray*}
Thus \eqref{hanti04} hold if and only if \eqref{cocycle4} hold.

Therefore we obtain that $\fa\oplus_\omega V$ is a Hom-Lie antialgebra if and only if $\omega=(\omega_0,\omega_1,\omega_2)$ is a 2-cocycle with coefficients in $V$.
We complete the proof.
\end{proof}

\begin{prop}\label{prop:crossedmod}
Let $(V,\a_{V_0},\b_{V_1})$, $(\fa,\a,\b)$ and $(\widetilde{\fa},\widetilde{\a},\widetilde{\b})$ be Hom-Lie antialgebras. Every central  extension of $\fa$ by $V$
\begin{equation*}
\xymatrix@C=0.5cm{
  0 \ar[r]^{} & V  \ar[rr]^{i } && \widetilde{\fa } \ar[rr]^{\pi} && \fa  \ar[r]^{} & 0 }
\end{equation*}
of Hom-Lie antialgebras give rise to a crossed module $(\widetilde{\fa}, \fa,\pi)$ where $\pi$ is the projective map.
\end{prop}

\begin{proof}
First, we prove that there is an action of $\fa$ on $\widetilde{\fa}$ in the natural way. By the above Lemma \ref{lem01},  we write   $\widetilde{\fa}=\fa\oplus V$ as direct sum space.
Define an action $\fa$ on $\widetilde{\fa}$  by
\begin{eqnarray*}
\rho_0(x_1)(x_2,u_2)&:=&(x_1\cdot  x_2,\omega_0(x_1, x_2)),\\
\rho_0(x_1)(y_1,v_1)&:=&(x_1\cdot y_1,\omega_1(x_1,y_1)),\\
\rho_1(y_1)(y_2,v_2)&:=&([y_1,y_2],\omega_2(y_1,y_2)).
\end{eqnarray*}
Since $\omega$ is a 2-cocycle, it is easy to verify that this is indeed an action of $\fa$ on $\widetilde{\fa}$.
Next, we are going to verify that this action satisfies the conditions in Definition \ref{def:crossedmod}.
For $(x_1,u_1)\in \fa_0\oplus V_0$, $(y_1,v_1)\in \fa_1\oplus V_1$, we have $\partial_0(x_1,u_1)=\pi_0(x_1,u_1)=x_1$, $\partial_1(y_1,v_1)=\pi_1(y_1,v_1)=y_1$.
Thus by direct computations we get
\begin{eqnarray*}
\partial_0\circ\rho_0(x_1)(x_2,u_2)&=&x_1\cdot\partial_0(x_2,u_2)=x_1\cdot  x_2,\\
\partial_1\circ\rho_0(x_1)(y_1,v_1)&=&x_1\cdot\partial_1((y_1,v_1))=x_1\cdot y_1,\\
\partial_1\circ\rho_1(y_1)(x_2,u_2)&=&y_1\cdot\partial_0(x_2,u_2)=y_1\cdot x_2,\\
\partial_0\circ\rho_1(y_1)(y_2,v_2)&=&[y_1,\partial_1(y_2,v_2)]=[y_1,y_2],\\
\rho_0(\partial_0(x_1,u_1))(x_2,u_2)&=&(x_1,u_1)\cdot (x_2,u_2)=(x_1\cdot  x_2,\omega_0(x_1, x_2)),\\
\rho_1(\partial_1(y_1,v_1))(y_2,v_2)&=&[(y_1,v_1),(y_2,v_2)]=([y_1,y_2],\omega_2(y_1,y_2)),\\
\rho_0(\partial_0(x_1,u_1))(y_1,v_1)&=&(x_1,u_1)\cdot(y_1,v_1)=(x_1\cdot y_1,\omega_1(x_1,y_1)).
\end{eqnarray*}
Therefore we obtain a crossed module $(\widetilde{\fa}, \fa,\pi)$.
\end{proof}

\begin{exe}\label{exe02}
Consider the 3-dimensional Hom-Lie antialgebra $\fa=\{\e; a_1, a_2\}$ under the following map:
$$\a(\e)=\e, \quad \b(a_1)=\mu a_1,\quad \b(a_2)=\mu^{-1} a_2,\quad {[a_1, a_2]}=\e.$$
It admit a  1-dimensional central  extension by $V=\{0; z\}, \b_{V_1}({z})=\mu {z}$ with the nontrivial 2-cocycle $\omega_1(\e, a_1)=\mu {z}$.
Thus we obtain a 4-dimensional Hom-Lie antialgebra
$\widetilde{\fa}=\{\e;a_1, a_2, {z}\}$, where $\e$ is even and $a_1, a_2, {z}$ are odd,
and the linear maps $(\a,\b):\widetilde{\fa}\to \widetilde{\fa}$ are defined by
$$\a(\e)=\e, \quad \b(a_1)=\mu a_1,\quad \b(a_2)=\mu^{-1} a_2,\quad \b({z})=\mu {z},$$
on the basis elements,  , together with structural operations
\begin{eqnarray*}
{[a_1,a_2]}=\e,\quad \e\cdot{}a_1=\mu {z}.
\end{eqnarray*}
By Proposition \ref{prop:crossedmod} we obtain a crossed module $(\widetilde{\fa}, \fa,\pi)$ where $\pi:\widetilde{\fa}\to \fa$ is given by
$$ \e\mapsto\e,\quad a_1\mapsto a_1,\quad  a_2\mapsto a_2,\quad  z\mapsto 0.$$


\end{exe}

\section{Universal central extensions}

In this section we will deal with universal central extensions
of  Hom-Lie antialgebras. We generalize classical results of universal central extensions theory
of Lie algebras. We also prove the existence result of the universal central extensions under the perfect conditions.

\begin{defi}
A central extension  \begin{equation*}
\xymatrix@C=0.5cm{
  0 \ar[r]^{} & V  \ar[rr]^{i } && \widetilde{\fa } \ar[rr]^{\pi} && \fa  \ar[r]^{} & 0 }\leqno (CE)
\end{equation*}
is said to be universal if for every   central extension \begin{equation*}
\xymatrix@C=0.5cm{
  0 \ar[r]^{} & V  \ar[rr]^{\widehat{i}} && \widehat{\fa } \ar[rr]^{\widehat{\pi}} && \fa  \ar[r]^{} & 0 }\leqno (CE')
\end{equation*}
there exists a  unique homomorphism of Hom-Lie antialgebras
$\varphi: \widetilde{\fa }\to \widehat{\fa }$ such that $\widehat{\pi}\circ\varphi = \pi$.
\end{defi}

\begin{defi}
A Hom-Lie antialgebra $(\fa,\a,\b)$ is said to be perfect if  $\fa_0 = \fa_0\cdot \fa_0=[\fa_1, \fa_1]$ and $\fa_1 = \fa_0\cdot \fa_1$.
\end{defi}

If we define the ideal of $\fa$ generated by the elements of the form $[[\fa,\fa]]=\langle x_1\cdot x_2, x_1\cdot y_1,  y_1\cdot y_2 |\forall x_1, x_2\in\fa_0, \forall y_1, y_2\in\fa_1\rangle$, then $\fa$ is perfect if and only if $\fa=[[\fa,\fa]]$.


\begin{lemma} \label{lemma1}
Let $\varphi:\widetilde{\fa}\rightarrow\fa$  be a surjective homomorphism  of  Hom-Lie antialgebras. If $(\widetilde{\fa},\widetilde{\a},\widetilde{\b})$ is a perfect  Hom-Lie antialgebra, then $(\fa,\a,\b)$ is also a perfect  Hom-Lie antialgebra.
\end{lemma}
\begin{proof}
Since $(\widetilde{\fa},\widetilde{\a},\widetilde{\b})$ is a perfect  Hom-Lie antialgebra and $\varphi:\widetilde{\fa}\rightarrow\fa$  be a surjective homomorphism, we get
 $\fa_0=\varphi(\widetilde{\fa}_0) = \varphi(\widetilde{\fa}_0\cdot \widetilde{\fa}_0)
 = \varphi(\widetilde{\fa}_0)\cdot \varphi(\widetilde{\fa}_0)=\fa_0\cdot \fa_0$. Similarly we have
$\fa_0 =\varphi([\widetilde{\fa}_1, \widetilde{\fa}_1])=[\varphi(\widetilde{\fa}_1), \varphi(\widetilde{\fa}_1)]=[\fa_1, \fa_1]$ and $\fa_1=\varphi(\widetilde{\fa}_1) = \varphi(\widetilde{\fa}_0\cdot \widetilde{\fa}_1)= \varphi(\widetilde{\fa}_0)\cdot \varphi(\widetilde{\fa}_1)=\fa_0\cdot \fa_1$. Thus $(\fa,\a,\b)$ is also a perfect  Hom-Lie antialgebra.
\end{proof}

\begin{lemma} \label{lemma2}
Let $\xymatrix@C=0.5cm{0 \ar[r]^{} & V  \ar[r]^{i } &\widetilde{\fa} \ar[r]^{\pi} & \fa  \ar[r]^{} & 0 }$ be a central extension and $\widetilde{\fa}$ be a perfect Hom-Lie antialgebra. Then there exists at most one homomorphism from (CE) to a second central extension of $\fa$.
\end{lemma}
\begin{proof}
 Assume there are two homomorphisms of Hom-Lie antialgebras
$\varphi, \varphi^{\prime}: \widetilde{\fa}\to \widehat{\fa}$. For $y_1,y_2\in \widetilde{\fa}_1$, we have
$$
\begin{aligned}\left(\varphi-\varphi^{\prime}\right)(x_1\cdot x_2)
&=\varphi(x_1\cdot x_2)-\varphi^{\prime}(x_1\cdot x_2) \\
&=\varphi(x_1)\cdot \varphi(x_2)-\varphi^{\prime}(x_1)\cdot \varphi^{\prime}(x_2) \\
&=(\varphi(x_1)-\varphi^{\prime}(x_1))\cdot \varphi(x_2)+\varphi^{\prime}(x_1)\cdot (\varphi(x_2)-\varphi^{\prime}(x_2))\\
&=0,
\end{aligned}
$$
$$
\begin{aligned}\left(\varphi-\varphi^{\prime}\right)(x_1\cdot y_1)
&=\varphi(x_1\cdot y_1)-\varphi^{\prime}(x_1\cdot y_1) \\
&=\varphi(x_1)\cdot \varphi(y_1)-\varphi^{\prime}(x_1)\cdot \varphi^{\prime}(y_1) \\
&=(\varphi(x_1)-\varphi^{\prime}(x_1))\cdot \varphi(y_1)+\varphi^{\prime}(x_1)\cdot (\varphi(y_1)-\varphi^{\prime}(y_1))\\
&=0,
\end{aligned}
$$
$$
\begin{aligned}\left(\varphi-\varphi^{\prime}\right)([y_1, y_2])
&=\varphi([y_1, y_2])-\varphi^{\prime}([y_1, y_2]) \\
&=[\varphi(y_1), \varphi(y_2)]-\left[\varphi^{\prime}(y_1), \varphi^{\prime}(y_2)\right] \\
&=\left[\varphi(y_1)-\varphi^{\prime}(y_1), \varphi(y_2)\right]+\left[\varphi^{\prime}(y_1), \varphi(y_2)-\varphi^{\prime}(y_2)\right]\\
&=0,
\end{aligned}
$$
where the last equality follows from the fact that $\varphi(x)-\varphi^{\prime}(x),\varphi(y)-\varphi^{\prime}(y)\in Z(\fa)$ for all $x\in \fa_0, y\in \fa_1$.
\end{proof}

Using classical arguments, we can see that a universal central extension, when it exists, is unique up to an isomorphism.

\begin{lemma} \label{lemma3}
If $\xymatrix@C=0.5cm{0 \ar[r]^{} & V  \ar[r]^{i } &\widetilde{\fa} \ar[r]^{\pi} & \fa  \ar[r]^{} & 0 }$ is a universal central extension, then  $\widetilde{\fa}$ and $\fa$ are perfect Hom-Lie antialgebras.
\end{lemma}
\begin{proof}
By Lemma \ref{lemma1} we have to proof that $\wt{\fa}$ is perfect.
Let us assume that $\wt{\fa}$ is not perfect, then $[[\wt{\fa},\wt{\fa}]] \varsubsetneq \wt{\fa}$. Let $W=\wt{\fa}/[[\wt{\fa},\wt{\fa}]]$. Let us consider the central extension $\xymatrix@C=0.5cm{0 \ar[r]^{} & W  \ar[r]^{i } &{W\oplus \fa} \ar[r]^{\wt{\pi}} & \fa  \ar[r]^{} & 0 }$ where $\wt{\pi}:W\oplus \fa \to \fa$ is the projection. Then we have the  homomorphisms of Hom-Lie antialgebras
$\varphi, \psi : \wt{\fa}\to W\oplus \fa$ given by $\varphi(x)=(x+[[\wt{\fa},\wt{\fa}]],\pi(x))$ and $\psi(x)=(0,\pi(x)), x\in \wt{\fa}$. One can verify that  $\wt{\pi}\circ \phi = \pi = \wt{\pi}\circ \psi$.
Thus $\xymatrix@C=0.5cm{0 \ar[r]^{} & V  \ar[r]^{i } &\widetilde{\fa} \ar[r]^{\pi} & \fa  \ar[r]^{} & 0 }$ cannot be a universal central extension. This is a contradiction.
\end{proof}

\begin{thm}\label{teorema}\
A Hom-Lie antialgebra $(\fa,\a,\b)$ admits a universal central extension if and only if $(\fa,\a,\b)$ is perfect.
Moreover, the kernel of the universal central extension is canonically isomorphic to the second homology group $H_2(\fa)$.
\end{thm}

\begin{proof}
 We already know that if the Hom-Lie algebra $\fa$ admit a universal central extension, then $\fa$ is perfect.
 Conversely, for a Hom-Lie antialgebra $(\fa,\a,\b)$, we construct a universal central extension as follows.

First, we construct a central extension for a given Hom-Lie algebra $\fa$.
Let $I_\fa$ be the  subspace of $\fa\otimes \fa$ spanned by the elements of the form
\begin{eqnarray*}
x_1\otimes x_2-x_2\otimes x_1,\ \a(x_1)\otimes x_2\cdot x_3-x_1\cdot x_2 \otimes \a(x_3)\in \fa_0\otimes \fa_0,\\
x_1\otimes y_1-y_1\otimes x_1,\ \a(x_1) \otimes  x_2\cdot y_1-\half x_1\cdot x_2 \otimes \b(y_1)\in\fa_0\otimes \fa_1,\\
y_1\otimes y_2+y_2\otimes y_1\in\fa_1\otimes \fa_1,\\
\a(x_1) \otimes [y_1,y_2]-x_1\cdot y_1  \otimes  \b(y_2)-\b(y_1) \otimes x_1\cdot y_2\in\fa_0\otimes \fa_0+\fa_1\otimes \fa_1,\\
\b(y_1)\otimes [y_2,y_3] +\b(y_2)\otimes[y_3,y_1]+\b(y_3)\otimes[y_1,y_2]\in\fa_1\otimes \fa_0.
\end{eqnarray*}
 That is, $I_\fa = {\rm Im} d_3$.

We denote the quotient space  $\mathfrak{U}=\frac{\fa \otimes \fa}{I_\fa}$  by $\frak{uce}(\fa)$. Elements in this quotient space of the form $x_1 \cdot x_2 + I_\fa$, $x_1 \cdot y_1 + I_\fa$ and $[y_1, y_2] + I_\fa$ are denoted by $x_1\ast x_2\in\frak{uce}(\fa)_0$, $x_1\ast y_1\in\frak{uce}(\fa)_1$ and $\{y_1,y_2\}\in\frak{uce}(\fa)_0$.

By construction, we have $x_1\ast x_2=x_2\ast x_1,\ x_1\ast y_1=y_1\ast x_1, \{y_1, y_2\}=-\{y_2, y_1\}$ and the following identities holds:
\begin{eqnarray}
\label{uce-cocyle1}
\a(x_1)\ast (x_2\cdot x_3)-(x_1\cdot x_2) \ast \a(x_3)=0,\\
\label{uce-cocyle2}
\a(x_1) \ast(x_2\cdot y_1)-\half (x_1\cdot x_2) \ast \b(y_1)=0,\\
\label{uce-cocyle3}
\a(x_1)\ast [y_1,y_2]-\{x_1\cdot y_1, \b(y_2)\}-\{\b(y_1), x_1\cdot y_2\}=0,\\
\label{uce-cocyle4}
\{\beta(y_1),[y_2,y_3]\} +  \{\beta(y_2),[y_3,y_1]\}+\{\beta(y_3),[y_1,y_2]\}=0,
\end{eqnarray}
for all $x_1, x_2, x_3 \in \fa_0$ and $y_1, y_2, y_3 \in \fa_1$.

Note that if we define $\omega_0:\fa_0\times \fa_0\to \frak{uce}(\fa)_0,\,\omega_1:\fa_0\times \fa_1\to \frak{uce}(\fa)_1,\,\omega_2:\fa_1\times \fa_1\to \frak{uce}(\fa)_0$ by $\omega_0(x_1,x_2)=x_1\ast x_2,\,\omega_1(x_1,y_1)=x_1\ast y_1,\,\omega_2(y_1,y_2)=\{y_1,y_2\}$,
then the above identities \eqref{uce-cocyle1}--\eqref{uce-cocyle4} means that $\omega=(\omega_0,\omega_1,\omega_2)$ is a 2-cocycle of $\fa$ with coefficients in $\frak{uce}(\fa)$.

Since $d_2(I_\fa)=d_2(d_3(\fa\otimes\fa\otimes\fa))=0$, so it induces a linear map $u: \frak{uce}(\fa) \to \fa$, given by
\begin{eqnarray*}
u_0(x_1\ast x_2)&=&x_1\cdot x_2,\\
u_1(x_1\ast y_1)&=&x_1\cdot y_1,\\
u_0(\{y_1,y_2\})&=&[y_1,y_2].
\end{eqnarray*}
Moreover, if we define $\widetilde{\alpha}, \widetilde{\beta}: \frak{uce}(\fa) \to \frak{uce}(\fa)$ by $\widetilde{\alpha}(x_1\ast x_2)=\alpha (x_1)\ast \alpha (x_2)$, $\widetilde{\alpha}(\{y_1,y_2\}) = \{\beta(y_1), \beta(y_2)\}$, $\widetilde{\beta}(x_1\ast y_1)={\alpha}(x_1)\ast {\beta}(y_1)$, then $(\frak{uce}(\fa), \widetilde{\alpha}, \widetilde{\beta})$ is a  Hom-Lie antialgebra with respect to the following structural operations:
\begin{eqnarray*}
 (x_1\ast x_2)\cdot (x_3\ast x_4)&=& (x_1\cdot x_2) \ast (x_3 \cdot x_4),\\
 \{y_1, y_2\}\cdot\{y_3,y_4\}&=&[y_1, y_2]\ast [y_3, y_4],\\
  (x_1\ast x_2)\cdot\{y_1,y_2\}&=& (x_1\cdot x_2)\ast [y_1,y_2],\\
  (x_1\ast x_2)\cdot(x_3\ast y_1)&=&(x_1\cdot x_3)\ast(x_3\cdot y_1),\\
  \{y_1,y_2\}\cdot (x_1\ast y_3)&=&[y_1,y_2]\ast (x_1\cdot y_3),\\
 {[x_1\ast y_1,x_2\ast y_2]}&=& \{x_1\cdot y_1, x_2\cdot y_2\}.
 \end{eqnarray*}
It is easy to see that $u  : \frak{uce}(\fa) \to \fa$ is a homomorphism of Hom-Lie antialgebras.
Since  the image of $u$ is contained in $(\fa_0\cdot \fa_0+\fa_0\cdot \fa_1)\oplus [\fa_1,\fa_1]$ and $(\fa,\alpha,\beta)$ is a perfect Hom-Lie antialgebra, so $u $ is a surjective homomorphism.

Second, we verify that this central extension is universal.
Now for any central extension  $\xymatrix@C=0.5cm{0 \ar[r]^{} & V  \ar[r]^{i } & U \ar[r]^{\pi} & \fa  \ar[r]^{} & 0 }$,
we choose a section $s: \fa \to U$ of $\pi$ such that $\pi\circ s = id_\fa$.
Using that $U$ is a Hom-Lie antialgebra and $s(x_1\cdot x_2)-s(x_1)\cdot s(x_2),\,
s(x_1\cdot y_1-s(x_1)\cdot s(y_1),\, s([y_1,y_2])-[s(y_1), s(y_2)]\in Z(U)$,
one verifies that this map annihilates $I_\fa$. Thus we obtain a linear map
$\varphi: {\mathfrak{uce}}(\fa) \rightarrow U$ by
$$
\begin{aligned}
 \varphi(x_{1}\ast x_{2})&= s\left(x_{1}\right)\cdot s\left(x_{2}\right),\\
 \varphi(x_{1}\ast y_{1})&= s\left(x_{1}\right)\cdot s\left(y_{1}\right),\\
\varphi(\left\{y_{1}, y_{2}\right\})&=\left[s\left(y_{1}\right), s\left(y_{2}\right)\right].
\end{aligned}
$$
By direct computations, we get
$$
\begin{aligned}
\varphi\big((x_{1}\ast x_{2})\cdot(x_{3}\ast x_{4})\big)
&=\varphi\big((x_{1}\cdot x_{2})\ast(x_{3}\cdot x_{4})\big)\\
&=s(x_{1}\cdot x_{2})\cdot s(x_{3}\cdot x_{4}) \\
&=\big(s(x_{1})\cdot s(x_{2})\big)\cdot \big(s(x_{3})\cdot s(x_{4})\big)\\
&=\varphi(x_{1}\ast x_{2})\cdot \varphi(x_{3}\ast x_{4}),
\end{aligned}
$$
$$
\begin{aligned}
\varphi\big((x_{1}\ast x_{2})\cdot\left\{ y_{1}, y_{2}\right\}\big)
&=\varphi\big((x_{1}\cdot x_{2})\ast\left[y_{1}, y_{2}\right]\big)\\
&=s(x_{1}\cdot x_{2})\cdot s([y_{1}, y_{2}])\\
&=\big(s(x_{1})\cdot s(x_{2})\big)\cdot\left[s\left(y_{1}\right), s\left(y_{2}\right)\right]\\
&=\varphi\left(x_{1}\ast x_{2}\right)\cdot \varphi\left(\left\{ y_{1}, y_{2}\right\}\right),
\end{aligned}
$$
$$
\begin{aligned}
\varphi\big([x_{1}\ast y_{1},x_{2}\ast y_{2}]\big)
&=\varphi\left\{x_1\cdot y_1, x_2\cdot y_2\right\}\\
&=\left[s(x_1\cdot y_1), s(x_2\cdot y_2)\right] \\
&=\left[s\left(x_{1}\right)\cdot s\left(y_{1}\right), s\left(x_{2}\right)\cdot s\left(y_{2}\right)\right]\\
&=[\varphi(x_{1}\ast y_{1}),\varphi(x_{2}\ast y_{2})].
\end{aligned}
$$
Similarly, we obtain
 $$\varphi\big(\{y_1, y_2\}\cdot\{y_3,y_4\}\big)=\varphi\big(\{y_1, y_2\}\big)\cdot \varphi\big(\{y_3,y_4\}\big),$$
 $$\varphi\big((x_1\ast x_2)\cdot(x_3\ast y_1)\big)=\varphi\big((x_1\ast x_2)\big)\cdot \varphi\big((x_3\ast y_1)\big),$$
  and
 $$\varphi\big(\{y_1,y_2\}\cdot (x_1\ast y_3)\big)=\varphi\big(\{y_1,y_2\}\big)\cdot \varphi\big((x_1\ast y_3)\big).$$
 Thus $\varphi: {\mathfrak{uce}}(\fa) \rightarrow U$ is a Hom-Lie antialgebra homomorphism.

Moreover, we have
$$\begin{aligned}
\pi\circ\varphi\big(x_{1}\ast x_{2}\big)
&= \pi\big(s\left(x_{1}\right)\cdot s\left(x_{2}\right)\big)\\
&=\pi\left(s\left(x_{1}\right)\right)\cdot \pi\left(s\left(x_{2}\right)\right)\\
&=x_{1}\cdot x_{2}=u\big(x_{1}\ast x_{2}\big),
\end{aligned}$$
$$\begin{aligned}
\pi\circ\varphi\big(x_{1}\ast y_{1}\big)
&= \pi\big(s\left(x_{1}\right)\cdot s\left(y_{1}\right)\big)\\
&=\pi\left(s\left(x_{1}\right)\right)\cdot \pi\left(s\left(y_{1}\right)\right)\\
&=x_{1}\cdot y_{1}=u\big(x_{1}\ast y_{1}\big),
\end{aligned}$$
$$\begin{aligned}
(\pi \circ \varphi)\big(\left\{ y_{1}, y_{2}\right\}\big)
&= \pi\big(\left[s\left(y_{1}\right), s\left(y_{2}\right)\right]\big)\\
&=\left[\pi\left(s\left(y_{1}\right)\right), \pi\left(s\left(y_{2}\right)\right)\right]\\
&=\left[y_{1}, y_{2}\right]=u\big(\left\{ y_{1}, y_{2}\right\}\big),
\end{aligned}$$
which implies $\pi\circ \varphi= {u}$.

Finally, since $(\frak{uce}(\fa), \widetilde{\alpha})$ is a perfect Hom-Lie antialgebra, so by Lemma \ref{lemma2} we have $\varphi$ is unique.

From the above construction, it follows that the kernal of $u$ is  $H_2(\fa)$, so we have the central extension
$$\xymatrix@C=0.5cm{0 \ar[r]^{} & H_2(\fa)  \ar[rr]^{i } && \frak{uce}(\fa) \ar[rr]^{\pi} && \fa  \ar[r]^{} & 0 }.$$
The proof is finished.
\end{proof}

It is interesting to study the lifting of automorphisms and derivations to central extensions, and to study the general theory of non-abelian tensor product of two Hom-Lie antialgebras. We leave these problems for further investigations.

\subsection*{Acknowledgements}
This research was supported by NSFC(11501179, 11961049).

\end{document}